\documentclass[12pt,a4paper,oneside,reqno]{amsart}

\usepackage{lmodern}
\usepackage[margin=2.5cm]{geometry}
\usepackage{setspace}
\usepackage{amssymb}
\usepackage{color}
\usepackage[normalem]{ulem} 
\usepackage{mathtools}
\usepackage{lineno}
\usepackage{hyperref}

\author{Gregory J. Galloway}
\address{Department of Mathematics, University of Miami, Coral Gables, FL, USA.}
\email{galloway@math.miami.edu}

\author{Abraão Mendes}
\address{Instituto de Matemática, Universidade Federal de Alagoas, Maceió, AL, Brazil.}
\email{abraao.mendes@im.ufal.br}

\title[Some Rigidity Results For Charged Initial Data Sets]{Some Rigidity Results For Charged Initial Data Sets}

\newtheorem{thm}{Theorem}[section]
\newtheorem{prop}[thm]{Proposition}
\newtheorem{lemma}[thm]{Lemma}

\theoremstyle{remark}

\DeclareMathOperator{\Ric}{Ric}

\DeclareMathOperator{\id}{id}

\DeclareMathOperator{\tr}{tr}
\DeclareMathOperator{\divergence}{div}

\newcommand{\A}{\mathcal{A}}

\newcommand{\R}{\mathbb{R}}

\newcommand{\p}{\partial}

\renewcommand{\div}{\divergence}

\newcommand{\q}{\mathfrak{q}}

\newcounter{mnotecount}

\begin{document}

\raggedbottom

\numberwithin{equation}{section}

\setstretch{1.2}

\begin{abstract}
In this note, we consider some initial data rigidity results concerning \linebreak marginally outer trapped surfaces (MOTS). As is well known, MOTS play an important role in the theory of black holes and, at the same time, are interesting spacetime analogues of minimal surfaces in Riemannian geometry. The main results presented here expand upon earlier works by the authors, specifically addressing initial data sets incorporating charge.
\end{abstract}

\maketitle

\section{Introduction}

Over the years, it has become clear that \textit{marginally outer trapped surfaces} (MOTS for short) are a fundamental concept in the realm of general relativity, particularly in the study of black holes and gravitational collapse. These surfaces serve as crucial tools for understanding the complex nature of spacetime in regions of intense gravitational fields.

From a mathematical perspective, the importance of the theory of MOTS can be demonstrated through, for example, their role in proving positive mass theorems (see \textit{e.g.} \cite{EicHuaLeeSch,GalLee,LeeLesUng2022,LeeLesUng}). Positive mass results are fundamental in general relativity, as they assert that the total mass of an isolated gravitational system is nonnegative, reflecting the physical reality that energy cannot be negative.

In this article, we aim to advance our understanding of the geometry of these crucial objects (the MOTS) and, perhaps, shed some light on the intricate nature of spacetime dynamics, black holes, and gravitational collapse.

In \cite{GalMen2018}, motivated in part by a result due to H.~Bray, S.~Brendle, and A.~Neves~\cite{BraBreNev}, the authors considered the case of a spherical MOTS in a $3$-dimensional initial data set and proved the rigidity result that we paraphrase as follows (definitions are given in Section~\ref{sec:pre}):

\begin{thm}[{\cite[Theorem~3.2]{GalMen2018}}]\label{thm.GalMen2018}
Let $(M,g,K)$ be a $3$-dimensional initial data set.\linebreak Assume that $(M,g,K)$ satisfies the energy condition, $\mu-|J|\ge c$, for some constant $c>0$. If $\Sigma$ is a closed weakly outermost and outer area-minimizing MOTS in $(M,g,K)$, then $\Sigma$ is topologically $S^2$ and its area satisfies
\begin{align*}
|\Sigma|\le\frac{4\pi}{c}.
\end{align*}
Furthermore, if equality holds,
then there exists an outer neighborhood $U\cong[0,\delta)\times\Sigma$ of $\Sigma$ in $M$ such that:
\begin{enumerate}
\item[\rm (a)] $(U,g)$ is isometric to $([0,\delta)\times\Sigma,dt^2+g_0)$, for some $\delta>0$, where $g_0$ - the induced metric on $\Sigma$ - has constant Gaussian curvature $\kappa=c$;
\item[\rm (b)] $K=fdt^2$ on $U$, where $f\in C^\infty(U)$ depends only on $t\in[0,\delta)$;
\item[\rm (c)] $\mu=c$ and $J=0$ on $U$.
\end{enumerate}
\end{thm}

Above, $\mu$ and $J$ are the \textit{local energy density} and the \textit{local current density} of $(M,g,K)$, respectively, defined in terms of the spacetime Einstein tensor $G$ by $\mu=G(u,u)$ and $J=G(u,\cdot)|_M$, where $u$ is the future directed timelike unit normal to $M$. The Gauss-Codazzi equations give that $\mu$ and $J$ are initial data quantities.

More recently, the authors obtained a global version of Theorem~\ref{thm.GalMen2018}, which, in particular, does not require the weakly outermost condition; see Theorem~\ref{thm.GalMen2024} below. This was inspired in part by their joint work with M.~Eichmair~\cite{EicGalMen}, where, among other things, they obtained a global version of the local rigidity result in~\cite{Gal2018} in connection with J.~Lohkamp's approach to the spacetime positive mass theorem in~\cite{Loh}.

\begin{thm}[{\cite[Theorem~3.1]{GalMen2024}}]
\label{thm.GalMen2024}
Let $(M,g,K)$ be a $3$-dimensional compact-with-\linebreak boundary initial data set. Assume that $(M,g,K)$ satisfies the energy condition, $\mu-|J|\ge c$, for some constant $c>0$. Assume also that the boundary of $M$ can be expressed as a \linebreak disjoint union $\p M=\Sigma_0\cup S$ of nonempty unions of components such that the following conditions hold:
\begin{enumerate}
\item $\theta^+\le0$ on $\Sigma_0$ with respect to the normal that points into $M$;
\item $\theta^+\ge0$ on $S$ with respect to the normal that points out of $M$;
\item $M$ satisfies the homotopy condition with respect to $\Sigma_0$;
\item the relative homology group $H_2(M,\Sigma_0)$ vanishes;
\item $\Sigma_0$ minimizes area. 
\end{enumerate}
Then $\Sigma_0$ is topologically $S^2$ and its area satisfies
\begin{align*}
\A(\Sigma_0)\le\frac{4\pi}{c}.
\end{align*}
Furthermore, if equality holds, then
\begin{enumerate}
\item[\rm (a)] $(M,g)$ is isometric to $([0,\ell]\times\Sigma_0,dt^2+g_0)$, for some $\ell>0$, where $g_0$ - the induced metric on $\Sigma_0$ - has constant Gaussian curvature $\kappa=c$;
\item[\rm (b)] $K=fdt^2$ on $M$, where $f\in C^\infty(M)$ depends only on $t\in[0,\ell]$;
\item[\rm (c)] $\mu=c$ and $J=0$ on $M$.
\end{enumerate}
\end{thm}

The aim of this paper is to extend Theorems~\ref{thm.GalMen2018} and~\ref{thm.GalMen2024} to spherical MOTS in\linebreak 3-dimensional charged initial data sets, that is, initial data for the Einstein-Maxwell equations with vanishing magnetic field. See \cite{BatLimSou} for some results related to Theorem~\ref{thm.GalMen2018}.

The paper is organized as follows: in Section~\ref{sec:pre}, we present some basic definitions; in Section~\ref{sec:local}, we prove an infinitesimal rigidity result (Proposition~\ref{prop.inf.rigidity}) for closed stable spherical MOTS in $3$-dimensional charged initial data sets and use this result to obtain an extension of Theorem~\ref{thm.GalMen2018}; this is the content of Theorem~\ref{thm.local.splitting}. Finally, in Section~\ref{sec:global}, we use the results obtained in the previous section to extend Theorem~\ref{thm.GalMen2024} to the context of MOTS in $3$-dimensional compact-with-boundary initial data sets for the Einstein-Maxwell equations with vanishing magnetic field; see Theorem~\ref{thm.global.splitting}. 

\medskip
\paragraph{\bf{Acknowledgements.}} The work of the first named author was partially supported by the Simons Foundation, under Award
No. 850541. The work of the second author was partially supported by the Conselho Nacional de Desenvolvimento Científico e Tecnológico - CNPq, Brazil (Grants 309867/2023-1 and 405468/2021-0) and the Fundação de Amparo à Pesquisa do Estado de Alagoas - FAPEAL, Brazil (Process E:60030.0000002254/2022).

\section{Preliminaries}\label{sec:pre}

We first discuss, based on natural physical considerations, the energy condition relevant to general non-time symmetric initial data sets that allows for the presence of an electric field, but not a magnetic field.

Let $(M,g,K)$ be an initial data set in a spacetime $(\bar{M},\bar{g}$), \textit{i.e.} $M$ is a spacelike hypersurface with induced metric $g$ and second fundamental form $K$ with respect to the future directed timelike unit normal to $M$. Assume $(\bar{M},\bar{g})$ satisfies the Einstein equation,
\begin{align*}
G + \Lambda \bar{g} = 8\pi (T_F + T),
\end{align*}
where $G$ is the Einstein tensor, $G = \Ric_{\bar{M}}-\frac{1}{2}R_{\bar{M}}\bar{g}$, $T_F$ is the electromagnetic energy-momentum tensor, and $T$ is the 
energy-momentum tensor associated to any nongravitational and nonelectromagnetic fields (\textit{e.g.} matter fields) that may be present. 

Let $u$ be the future directed timelike unit normal vector field along $M$. As is standard, by the Gauss-Codazzi equations, we have
\begin{align*}
\mu:= G(u,u) = \frac{1}{2}\left(R_{M}-|K|^2+\tau^2\right)\quad\mbox{and}\quad J := G(u,\cdot) 
=\div(K-\tau g) ,
\end{align*}
where $R_{M}$ is the scalar curvature of $M$ and $\tau=\tr K$ is the mean curvature of $M$ in $\bar{M}$ with respect to $u$.

Suppose that the fields associated to $T$ satisfy the \textit{dominant energy condition}, that is,
\begin{align*}
T(X,Y) \ge 0\quad\mbox{for all future directed causal vectors}\quad X,Y.
\end{align*}
Then, for all unit vectors $\nu$ that are tangent to $M$,
\begin{align*}
G(u,u+\nu) + \Lambda\bar{g}(u,u+\nu) = 8\pi(T_F(u,u + \nu) + T(u,u+\nu)),
\end{align*}
and hence,
\begin{align*}
\mu + J(\nu) = G(u,u) + G(u,\nu) \ge \Lambda + 8\pi(T_F(u,u) + T_F(u,\nu)).
\end{align*}

Now suppose that there is no magnetic field present, $B=0$. It then follows that $T_F(u,\nu)=0$. Physically, this corresponds to the vanishing of the Poynting vector in every Lorentz orthogonal frame. Moreover, $T_F(u,u)$ represents the electromagnetic energy density and is given by the expression $\frac{1}{8\pi}|E|^2$, where $E$ is the electric field. We thus arrive at the initial data inequality, 
\begin{align}\label{chargeDEC}
\mu+J(\nu)\ge\Lambda+|E|^2,
\end{align}
which holds at each point $p \in M$ and for each unit vector $\nu\in T_pM$. As this holds for all such vectors $\nu$, \eqref{chargeDEC} in turn implies, 
\begin{align*}
\mu-|J|\ge\Lambda+|E|^2,
\end{align*}
which we refer to as the \textit{charged dominant energy condition.} 
(Essentially the same energy condition is used in \cite{BatLimSou}, but is arrived at by a different rationale. \textit{N.B.} however, we note that the quantity $\mu$ is used in a different way.)

Now let $\Sigma$ be a closed embedded hypersurface in $M$.

In this paper, we assume that $M$ and $\Sigma$ are orientable; in particular, $\Sigma$ is two-sided. Then, we fix a unit normal vector field $\nu$ along $\Sigma$; if $\Sigma$ separates $M$, by convention, we say that $\nu$ points to the \textit{outside} of $\Sigma$. 

In the sequel, we are going to present some important definitions to our purposes.

The \textit{charge} of $\Sigma$ with respect to $\nu$ is defined by
\begin{align*}
\q(\Sigma)=\frac{1}{4\pi}\int_\Sigma\langle E,\nu\rangle.
\end{align*}

The \textit{null second fundamental forms} $\chi^+$ and $\chi^-$ of $\Sigma$ in $(M,g,K)$ are defined by
\begin{align*}
\chi^+=K|_\Sigma+A\quad\mbox{and}\quad\chi^-=K|_\Sigma-A,
\end{align*}
where $A$ is the second fundamental form of $\Sigma$ in $(M,g)$ with respect to $\nu$; more precisely,
\begin{align*}
A(X,Y)=g(\nabla_X\nu,Y)\quad\mbox{for}\quad X,Y\in\mathfrak{X}(\Sigma),
\end{align*}
where $\nabla$ is the Levi-Civita connection of $(M,g)$.

The \textit{null expansion scalars} or the \textit{null mean curvatures} $\theta^+$ and $\theta^-$ of $\Sigma$ in $(M,g,K)$ with respect to $\nu$ are defined by
\begin{align*}
\theta^+=\tr\chi^+=\tr_\Sigma K+H\quad\mbox{and}\quad\theta^-=\tr\chi^-=\tr_\Sigma K-H,
\end{align*}
where $H=\tr A$ is the mean curvature of $\Sigma$ in $(M,g)$ with respect to $\nu$. Observe that $\theta^\pm=\tr\chi^\pm$.

After R.~Penrose, $\Sigma$ is said to be \textit{trapped} if both $\theta^+$ and $\theta^-$ are negative. Restricting our attention to one side, we say that $\Sigma$ is \textit{outer trapped} if $\theta^+$ is negative and \textit{marginally outer trapped} if $\theta^+$ vanishes. In the latter case, we refer to $\Sigma$ as a \textit{marginally outer trapped surface} or a \textit{MOTS}, for short.

Assume that $\Sigma$ is a MOTS in $(M,g,K)$ with respect to a unit normal $\nu$ that is a boundary in $M$, \textit{i.e.} $\nu$ points towards a top-dimensional submanifold $M^+\subset M$ such that $\p M^+=\Sigma\sqcup S$, where $S$ (possibly $S=\varnothing$) is a union of components of $\p M$ (in particular, if $\Sigma$ separates $M$). We think of $M^+$ as the region outside of $\Sigma$. Then we say that 
$\Sigma$ is \textit{outermost} (resp. \textit{weakly outermost}) if there is no closed embedded hypersurface in $M^+$ with $\theta^+\le0$ (resp. $\theta^+<0$) that is homologous to and different from $\Sigma$.

We say that $\Sigma$ \textit{minimizes area} in $M$ if $\Sigma$ has the least area in its homology class in $M$, \textit{i.e.} $|\Sigma|\le|\Sigma'|$ for every closed embedded hypersurface $\Sigma'$ in $M$ that is homologous to $\Sigma$. Similarly, $\Sigma$ is said to be \textit{outer area-minimizing} if $\Sigma$ minimizes area in $M^+$.

An important notion that we are going to recall now is the notion of stability for MOTS.

Let $\Sigma$ be a MOTS in $(M,g,K)$ with respect to $\nu$ and $t\to\Sigma_t$ be a variation of $\Sigma=\Sigma_0$ in $M$ with variation vector field $\frac{\p}{\p t}|_{t=0}=\phi\nu$, for some $\phi\in C^\infty(\Sigma)$. Denote by $\theta^\pm(t)$ the null expansion scalars of $\Sigma_t$ with respect to the unit normal $\nu_t$, where $\nu=\nu_t|_{t=0}$. It is well known that
\begin{align*}
\frac{\p\theta^+}{\p t}\Big|_{t=0}=-\Delta\phi+2\langle X,\nabla\phi\rangle+(Q-|X|^2+\div X)\phi,
\end{align*}
where $\Delta$ and $\div$ are the Laplace and divergent operators of $\Sigma$ with respect to the induced metric $\langle\,,\,\rangle$, respectively; $X\in\mathfrak{X}(\Sigma)$ is the vector field that is dual to the 1-form $K(\nu,\cdot)|_\Sigma$, and 
\begin{align*}
Q=\frac{1}{2}R_\Sigma-(\mu+J(\nu))-\frac{1}{2}|\chi^+|^2.
\end{align*}
Here $R_\Sigma$ represents the scalar curvature of $\Sigma$.

The operator
\begin{align*}
L\phi=-\Delta\phi+2\langle X,\nabla\phi\rangle+(Q-|X|^2+\div X)\phi,\quad\phi\in C^\infty(\Sigma),
\end{align*}
is called the \textit{MOTS stability operator}. It can be proved that $L$ has a real eigenvalue $\lambda_1$, called the \textit{principal eigenvalue} of $L$, such that $\mbox{Re}\,\lambda\ge\lambda_1$ for any complex eigenvalue $\lambda$. Furthermore, the associated eigenfunction $\phi_1$, $L\phi_1=\lambda_1\phi_1$, is unique up to scale and can be chosen to be everywhere positive.

We say that $\Sigma$ is \textit{stable} if $\lambda_1\ge0$; this is equivalent to saying that $L\phi\ge0$ for some positive function $\phi\in C^\infty(\Sigma)$. It is not difficult to see that if $\Sigma$ is weakly outermost (in particular, if $\Sigma$ is outermost), then $\Sigma$ is stable.

Before finishing this section, we would like to point out that, in general (when $\Sigma$ is not necessarily a MOTS), the first variation of $\theta^+$ is given by
\begin{align*}
\frac{\p\theta^+}{\p t}\Big|_{t=0}=-\Delta\phi+2\langle X,\nabla\phi\rangle+\left(Q-|X|^2+\div X-\frac{1}{2}\theta^++\tau\theta^+\right)\phi.
\end{align*}

\section{Infinitesimal Rigidity and Local Splitting}\label{sec:local}

The aim of this section is to generalize the local splitting result, Theorem \ref{thm.GalMen2018} in the introduction, to initial data sets with charge.  The first step is to establish an {\it infinitesimal rigidity} result; \textit{cf.} \cite[Proposition~1.1]{BatLimSou}.

\begin{prop}[Infinitesimal Rigidity]\label{prop.inf.rigidity}
Let $(M,g,K,E)$ be a $3$-dimensional initial data set for the Einstein-Maxwell equations with vanishing magnetic field, $B=0$. Assume that $(M,g,K,E)$ satisfies the charged dominant energy condition, $\mu-|J|\ge\Lambda+|E|^2$, for some constant $\Lambda>0$. If $\Sigma$ is a closed stable MOTS in $(M,g,K)$, then $\Sigma$ is topologically $S^2$ and its charge and area satisfy
\begin{align}\label{eq.charge}
4\Lambda\q(\Sigma)^2\le1,
\end{align}
\begin{align}\label{eq.area}
2\pi\left(1-\sqrt{1-4\Lambda\q(\Sigma)^2}\right)\le\Lambda|\Sigma|\le2\pi\left(1+\sqrt{1-4\Lambda\q(\Sigma)^2}\right).
\end{align}
Furthermore, if equality holds in one of the above inequalities, then
\begin{enumerate}
\item[\rm (a)] $E=a\nu$ on $\Sigma$, for some constant $a$;
\item[\rm (b)] $\Sigma$ is isometric to the round $2$-sphere of Gaussian curvature $\kappa=\Lambda+a^2$;
\item[\rm (c)] the null second fundamental form $\chi^+$ of $\Sigma$ vanishes;
\item[\rm (d)] $\mu+J(\nu)=\mu-|J|=\Lambda+a^2$ on $\Sigma$ and, in particular, $J|_\Sigma=0$;
\item[\rm (e)] the principal eigenvalue $\lambda_1$ of $L$ equals zero.
\end{enumerate}
\end{prop}

Observe that, if equality in \eqref{eq.charge} holds, then $\Lambda|\Sigma|=2\pi$ and equalities in \eqref{eq.area} also hold.

\proof
Let $\phi>0$ be an eigenfunction associated to the principal eigenvalue $\lambda_1$ of $L$,
\begin{align*}
\lambda_1\phi=L\phi=-\Delta\phi+2\langle X,\nabla\phi\rangle+(Q-|X|^2+\div X)\phi.
\end{align*}
Thus, dividing by $\phi$, we get
\begin{align}\label{eq.aux5}
0\le\lambda_1=\div Y-|Y|^2+Q\le\div Y+Q,
\end{align}
where $Y=X-\nabla\ln\phi$. Therefore, integrating over $\Sigma$ and using the Divergence Theorem, we obtain
\begin{align}\label{eq.aux1}
0\le\int_\Sigma Q=\int_\Sigma\left(\kappa-(\mu+J(\nu))-\frac{1}{2}|\chi^+|^2\right)\le4\pi(1-\textbf{g}(\Sigma))-\int_\Sigma(\Lambda+|E|^2),
\end{align}
where above we have used the Gauss-Bonnet Theorem and inequalities
\begin{align}\label{eq.aux6}
\mu+J(\nu)\ge\mu-|J|\ge\Lambda+|E|^2.
\end{align}
Here $\textbf{g}(\Sigma)$ represents the genus of $\Sigma$. Inequalities in \eqref{eq.aux1} give that $\textbf{g}(\Sigma)=0$, since $\Lambda>0$.

On the other hand, 
\begin{align}\label{eq.aux2}
\q(\Sigma)^2=\left(\frac{1}{4\pi}\int_\Sigma\langle E,\nu\rangle\right)^2\le\frac{1}{(4\pi)^2}\left(\int_\Sigma|E|\right)^2\le\frac{|\Sigma|}{(4\pi)^2}\int_\Sigma|E|^2.
\end{align}
Then, using \eqref{eq.aux2} into \eqref{eq.aux1}, we obtain
\begin{align}\label{eq.aux9}
\Lambda|\Sigma|+\frac{(4\pi\q(\Sigma))^2}{|\Sigma|}\le4\pi(1-\textbf{g}(\Sigma))=4\pi,
\end{align}
that is,
\begin{align*}
\Lambda|\Sigma|^2-4\pi|\Sigma|+(4\pi\q(\Sigma))^2\le0.
\end{align*}
Completing the square in the last inequality, we get
\begin{align}\label{eq.aux4}
(\Lambda|\Sigma|-2\pi)^2\le(2\pi)^2(1-4\Lambda\q(\Sigma)^2).
\end{align}
This gives \eqref{eq.charge} and \eqref{eq.area}.

If equality in \eqref{eq.charge} or \eqref{eq.area} holds, then we have equality in \eqref{eq.aux4}, which gives that all above inequalities must be equalities. Therefore,
\begin{itemize}
\item equalities in \eqref{eq.aux5} give that $\lambda_1=0$, $Y=0$ and, in particular, $Q=0$;
\item equalities in \eqref{eq.aux1} and \eqref{eq.aux6} give that $\chi^+=0$;
\item equalities in \eqref{eq.aux2} give that $E$ is parallel to $\nu$ and $|E|$ is constant; thus $E=a\nu$, for some constant $a$ (here we have used that $\Sigma$ is connected);
\item equalities in \eqref{eq.aux6}, together with $E=a\nu$, give that $\mu+J(\nu)=\mu-|J|=\Lambda+a^2$. 
\end{itemize}
Finally,
\begin{align*}
0=Q=\kappa-(\mu+J(\nu))-\frac{1}{2}|\chi^+|^2=\kappa-(\Lambda+a^2),\quad\mbox{that is,}\quad\kappa=\Lambda+a^2.  
\end{align*}

\vspace{-.33in}
\qed

\medskip

We will now make use of  Infinitesimal Rigidity to prove the following Local Splitting.  A key feature of the latter result (in comparison with the $\Lambda > 0$ case of Theorem 1.2 in \cite{BatLimSou}) is that it does not assume a convexity condition on $K$, specifically that $K$ be $2$-convex; this is the requirement that the sum of the two smallest eigenvalues of $K$ be nonnegative.

\begin{thm}[Local Splitting]\label{thm.local.splitting}
Let $(M,g,K,E)$ be a $3$-dimensional initial data set for the Einstein-Maxwell equations with vanishing magnetic field, $B=0$. Assume that $(M,g,K,E)$ satisfies the charged dominant energy condition, 
$\mu-|J|\ge\Lambda+|E|^2$, for some constant $\Lambda>0$, and $\div E=0$. If $\Sigma$ is a closed weakly outermost and outer area-minimizing MOTS in $(M,g,K)$, then $\Sigma$ is topologically $S^2$ and its charge and area satisfy
\begin{align}\label{eq.charge2}
4\Lambda\q(\Sigma)^2\le1,
\end{align}
\begin{align}\label{eq.area2}
\Lambda|\Sigma|\le 2\pi\left(1+\sqrt{1-4\Lambda\q(\Sigma)^2}\right).
\end{align}
Furthermore, if equality holds in one of the above inequalities,
then there exists an outer neighborhood $U\cong[0,\delta)\times\Sigma$ of $\Sigma$ in $M$ such that
\begin{enumerate}
\item[\rm (a)] $E=a\nu_t$, for some constant $a$, where $\nu_t$ is the unit normal to $\Sigma_t\cong\{t\}\times\Sigma$ in direction of the foliation;
\item[\rm (b)] $(U,g)$ is isometric to $([0,\delta)\times\Sigma,dt^2+g_0)$, for some $\delta>0$, where $g_0$ - the induced metric on $\Sigma$ - has constant Gaussian curvature $\kappa=\Lambda+a^2$;
\item[\rm (c)] $K=fdt^2$ on $U$, where $f\in C^\infty(U)$ depends only on $t\in[0,\delta)$;
\item[\rm (d)] $\mu=\Lambda+a^2$ and $J=0$ on $U$.
\end{enumerate}
\end{thm}

By Proposition \ref{prop.inf.rigidity}, if equality in \eqref{eq.charge2} holds, then $\Lambda|\Sigma|=2\pi$ and equality in \eqref{eq.area2} also holds.

\begin{proof}
Since $\Sigma$ is weakly outermost (in particular, stable), we may apply the Infinitesimal Rigidity (Proposition~\ref{prop.inf.rigidity}) to obtain that $\Sigma$ is topologically $S^2$ and its charge and area satisfy \eqref{eq.charge2} and \eqref{eq.area2}. Furthermore, if equality in \eqref{eq.charge2} or \eqref{eq.area2} holds, then conditions\linebreak (a)-(e) in the Infinitesimal Rigidity also hold. 
In particular, as a consequence of the fact that the principal eigenvalue $\lambda_1$ of $L$ equals zero, there exists a foliation of an outer neighborhood $U\cong[0,\delta)\times\Sigma$ of $\Sigma$ in $M$ by constant null mean curvature surfaces\linebreak $\Sigma_t\cong\{t\}\times\Sigma$ (see \cite[Lemma~2.3]{Gal2018}), with $\Sigma_0=\Sigma$, such that 
\begin{align*}
g=\phi^2dt^2+g_t\quad\mbox{on}\quad U,
\end{align*}
where $g_t$ is the induced metric on $\Sigma_t$. Recall that, on $\Sigma_t$,
\begin{align*}
\frac{d\theta}{dt}=-\Delta\phi+2\langle X,\nabla\phi\rangle+\left(Q-|X|^2+\div X-\frac{1}{2}\theta^2+\tau\theta\right)\phi,
\end{align*}
where $\theta=\theta(t)$ is the null mean curvature of $\Sigma_t$ with respect to $\nu_t=\phi^{-1}\p_t$. Therefore, dividing both sides of the above equation by $\phi$ and integrating over $\Sigma_t$, we obtain
\begin{align}
\theta'\int_{\Sigma_t}\frac{1}{\phi}-\theta\int_{\Sigma_t}\tau&=\int_{\Sigma_t}\left(\div Y-|Y|^2+Q-\frac{1}{2}\theta^2\right)\le\int_{\Sigma_t}Q\nonumber\\
&=\int_{\Sigma_t}\left(\kappa_t-(\mu+J(\nu_t))-\frac{1}{2}|\chi_t^+|^2\right)\label{eq.aux7}\\
&\le4\pi-\int_{\Sigma_t}(\Lambda+|E|^2),\nonumber
\end{align}
where $Y=X-\nabla\ln\phi$ on $\Sigma_t$, and $\kappa_t$ and $\chi_t^+$ represent the Gaussian curvature and the null second fundamental form of $\Sigma_t$, respectively. Above, we have used the Divergence and the Gauss-Bonnet Theorems and inequalities
\begin{align*}
\mu+J(\nu_t)\ge\mu-|J|\ge\Lambda+|E|^2.
\end{align*}
As in \eqref{eq.aux2}, we can see that
\begin{align}\label{eq.aux8}
\q(\Sigma_t)^2\le\frac{|\Sigma_t|}{(4\pi)^2}\int_{\Sigma_t}|E|^2.
\end{align}
Therefore, using \eqref{eq.aux8} into \eqref{eq.aux7}, we get
\begin{align*}
\theta'\int_{\Sigma_t}\frac{1}{\phi}-\theta\int_{\Sigma_t}\tau\le4\pi-\left(\Lambda|\Sigma_t|+\frac{(4\pi\q(\Sigma_t))^2}{|\Sigma_t|}\right).
\end{align*}

On the other hand, the Divergence Theorem gives that $\q(\Sigma_t)=\q(\Sigma_0)=:\q_0$ for all $t\in[0,\delta)$, since $\div E=0$. Then, defining the function
\begin{align*}
h(t)=\Lambda|\Sigma_t|+\frac{(4\pi\q_0)^2}{|\Sigma_t|},
\end{align*}
we have that $h(0)=4\pi$ (this follows from equality in \eqref{eq.aux9}) and
\begin{align*}
\theta'\int_{\Sigma_t}\frac{1}{\phi}-\theta\int_{\Sigma_t}\tau\le h(0)-h(t).
\end{align*}
Also, because $|\Sigma_0|=2\pi(1+\sqrt{1-4\Lambda\q_0})/\Lambda$ is the largest root of the polynomial
\begin{align*}
p(x)=\Lambda x^2-4\pi x+(4\pi\q_0)^2
\end{align*}
and $|\Sigma_0|\le|\Sigma_t|$, since $\Sigma_0=\Sigma$ is outer area-minimizing, we have 
\begin{align*}
0\le p(|\Sigma_t|)=\Lambda|\Sigma_t|^2-4\pi|\Sigma_t|+(4\pi\q_0)^2,
\end{align*}
that is,
\begin{align*}
h(0)=4\pi\le\Lambda|\Sigma_t|+\frac{(4\pi\q_0)^2}{|\Sigma_t|}=h(t).
\end{align*}
This gives that
\begin{align*}
\theta'\int_{\Sigma_t}\frac{1}{\phi}-\theta\int_{\Sigma_t}\tau\le h(0)-h(t)\le0.
\end{align*}
Taking $\alpha(t)=\int_{\Sigma_t}\tau/\int_{\Sigma_t}\frac{1}{\phi}$, we get 
\begin{align*}
\left(\theta(t)e^{-\int_0^t\alpha(s)ds}\right)'\le0,
\end{align*}
which implies
\begin{align*}
\theta(t)e^{-\int_0^t\alpha(s)ds}\le\theta(0)=0\quad\mbox{for all}\quad t\in[0,\delta).
\end{align*}
Since we are assuming that $\Sigma_0=\Sigma$ is weakly outermost, we obtain that $\theta(t)=0$ for all $t\in[0,\delta)$. In particular, all above inequalities must be equalities. Computations as in the proof of \cite[Theorem~3.2]{GalMen2018} give conclusions (b)-(d) of the theorem. By equality in \eqref{eq.aux8}, we get that $E=a\nu_t$ on $\Sigma_t$, for some function $a$ that depends only on $t\in[0,\delta)$. Finally, using that $\Sigma_t$ is minimal and $\div E=0$, we can see that $a$ is constant. This finishes the proof of the theorem.
\end{proof}

\subsection{The charged Nariai spacetime}
The charged Nariai spacetime $(\bar{M},\bar{g})$ is an exact solution to the (source-free) Einstein-Maxwell equations,
\begin{align*}
\left\{
\begin{array}{l}
G+\Lambda\bar{g}=8\pi T_F,\\
d F=0,\quad\div_{\bar{g}}F=0,
\end{array}
\right.
\end{align*}
where $F$ is a differential 2-form on $\bar{M}$ and $T_F$ is the electromagnetic energy-momentum tensor given by
\begin{align*}
T_F=\frac{1}{4\pi}\left(F\circ F-\frac{1}{4}|F|_{\bar{g}}^2\,\bar{g}\right),\quad (F\circ F)_{\alpha\beta}=\bar{g}^{\mu\nu}F_{\alpha\mu}F_{\beta\nu}.
\end{align*}
This is what is called an \textit{electrovacuum solution} or \textit{electrovacuum spacetime} (see Section~3 of~\cite{CruzLimaSousa} for a list of electrovacuum solutions, including the charged Nariai spacetime and the ultracold black hole system). In static coordinates,
\begin{align*}
\bar{g}=-V^2dt^2+ds^2+\rho^2d\Omega^2,\quad\bar{M}=\R\times(0,\pi/\alpha)\times S^2,
\end{align*}
where $V(s)=\sin(\alpha s)$, $d\Omega^2$ is the round metric on $S^2$ of constant Gaussian curvature one, and $\alpha>0$ and $\rho^2>0$ are suitable constants. More precisely, given parameters $m>0$, $\Lambda>0$ and $\q\in\R$, representing the mass, the cosmological constant and the electric charge, respectively, satisfying
\begin{align*}
4\Lambda\q^2<1\quad\mbox{and}\quad m^2=\frac{1}{18\Lambda}\left[1+12\Lambda\q^2+(1-4\Lambda\q^2)^{3/2}\right],
\end{align*}
$\rho^2$ is defined as the only solution to 
\begin{align*}
\rho^2(1-\Lambda\rho^2)=\q^2\quad\mbox{such that}\quad\frac{1}{2\Lambda}<\rho^2\le\frac{1}{\Lambda},
\end{align*}
that is,
\begin{align*}
\rho^2=\frac{1+\sqrt{1-4\Lambda\q^2}}{2\Lambda},\quad\mbox{and}\quad\alpha:=\sqrt{\Lambda-\frac{\q^2}{\rho^4}}.
\end{align*}
The 2-form $F$ is given by
\begin{align*}
F=-\frac{\q\sin(\alpha s)}{\rho^2}dt\wedge ds
\end{align*}
and, on a $t$-slice $M$ of $\bar{M}$, the electric field is given by
\begin{align*}
E=\dfrac{\q}{\rho^2}\p_s.
\end{align*}
Straightforward computations show that $M$ is totally geodesic in $(\bar{M},\bar{g})$, i.e. $K=0$.

Now, defining $a=\q/\rho^2$, we have
\begin{align*}
a^2=\frac{\q^2}{\rho^4}=\frac{1}{\rho^2}-\Lambda,\quad\mbox{that is,}\quad\rho^2=\frac{1}{\Lambda+a^2}.
\end{align*}
Therefore, the metric $\rho^2d\Omega^2$ equals to metric $g_0$ of constant Gaussian curvature $\kappa=\Lambda+a^2$. Thus, the induced metric $g$ on $M=(0,\pi/\alpha)\times S^2$ can be written as
\begin{align*}
g=ds^2+g_0.
\end{align*}

On the other hand, the area of $\Sigma=(S^2,g_0)=(S^2,\rho^2d\Omega^2)$ satisfies
\begin{align*}
\Lambda|\Sigma|=4\pi\Lambda\rho^2=2\pi\left(1+\sqrt{1-4\Lambda\q^2}\right).
\end{align*} 
This shows that the $t$-slices of the charged Nariai spacetime satisfy the hypotheses of Theorem~\ref{thm.local.splitting} when $4\Lambda\q^2<1$ and \eqref{eq.area2} is saturated. When $\q=0$, we have the standard Nariai spacetime.

In a similar fashion, the $t$-slices of the ultracold black hole system satisfy the hypotheses of Theorem~\ref{thm.local.splitting} when $4\Lambda\q^2=1$ and $\Lambda|\Sigma|=2\pi$.

\section{Compact-with-boundary initial data sets}\label{sec:global}

We now establish conditions under which the local splitting result can be extended to a global splitting. 

\begin{thm}\label{thm.global.splitting}
Let $(M,g,K,E)$ be a $3$-dimensional compact-with-boundary initial data set for the Einstein-Maxwell equations with vanishing magnetic field, $B=0$. Assume that $(M,g,K,E)$ satisfies the charged dominant energy condition, $\mu-|J|\ge\Lambda+|E|^2$, for some constant $\Lambda>0$, and $\div E=0$. Assume also that the boundary of $M$ can be expressed as a disjoint union $\p M=\Sigma_0\cup S$ of nonempty unions of components such that the following conditions hold:
\begin{enumerate}
\item $\theta^+\le0$ on $\Sigma_0$ with respect to the normal that points into $M$;
\item $\theta^+\ge0$ on $S$ with respect to the normal that points out of $M$;
\item $M$ satisfies the homotopy condition with respect to $\Sigma_0$;
\item the relative homology group $H_2(M,\Sigma_0)$ vanishes;
\item $\Sigma_0$ minimizes area. 
\end{enumerate}
Then $\Sigma_0$ is topologically $S^2$ and its charge and area satisfy
\begin{align}\label{eq.charge3}
4\Lambda\q(\Sigma_0)^2\le1,
\end{align}
\begin{align}\label{eq.area4}
\Lambda|\Sigma_0|\le 2\pi\left(1+\sqrt{1-4\Lambda\q(\Sigma_0)^2}\right).
\end{align}
Furthermore, if equality holds in one of the above inequalities, then $M\cong[0,\ell]\times\Sigma_0$, for some $\ell>0$, and
\begin{enumerate}
\item[\rm (a)] $E=a\nu_t$, for some constant $a$, where $\nu_t$ is the unit normal to $\Sigma_t\cong\{t\}\times\Sigma_0$ in direction of the foliation;
\item[\rm (b)] $(M,g)$ is isometric to $([0,\ell]\times\Sigma_0,dt^2+g_0)$, where $g_0$ - the induced metric on $\Sigma_0$ - has constant Gaussian curvature $\kappa=\Lambda+a^2$;
\item[\rm (c)] $K=fdt^2$ on $M$, where $f\in C^\infty(M)$ depends only on $t\in[0,\ell]$;
\item[\rm (d)] $\mu=\Lambda+a^2$ and $J=0$ on $M$.
\end{enumerate}
If equality in \eqref{eq.charge3} holds, then $\Lambda|\Sigma_0|=2\pi$ and equality in \eqref{eq.area4} also holds.
\end{thm}

As in Theorem \ref{thm.GalMen2024} in the introduction, the weakly outermost assumption is not required. We do, however, impose the same topological assumptions (3) and (4). By definition, $M$ satisfies the \textit{homotopy condition} with respect to $\Sigma\subset M$ provided there exists a continuous map $\rho:M\to\Sigma$ such that $\rho\circ i:\Sigma\to\Sigma$ is homotopic to $\id_\Sigma$, where $i:\Sigma\hookrightarrow M$ is the inclusion map (for example, if $\Sigma$ is a retract of $M$).

The following is the key lemma that enables us to apply the local splitting result.

\begin{lemma}\label{main.lemma}
Under the assumptions of Theorem \ref{thm.global.splitting}, we have that
\begin{enumerate}
\item[(a)] $4\Lambda\q(\Sigma_0)^2\le1$ and
\item[(b)] $\Sigma_0$ is a weakly outermost MOTS whose area satisfies
\begin{align}\label{eq.equality.area}
\Lambda|\Sigma_0|=2\pi\left(1+\sqrt{1-4\Lambda\q(\Sigma_0)^2}\right)
\end{align}
unless
\begin{align*}
\Lambda|\Sigma_0|<2\pi\left(1+\sqrt{1-4\Lambda\q(\Sigma_0)^2}\right).
\end{align*}
Furthermore, the above inequality cannot happen if $4\Lambda\q(\Sigma_0)^2=1$.
\end{enumerate}
\end{lemma}

\begin{proof}
Assume, for the moment, that (a) is true. Suppose 
\begin{align*}
\Lambda|\Sigma_0|\ge2\pi\left(1+\sqrt{1-4\Lambda\q(\Sigma_0)^2}\right).
\end{align*}
We are going to prove that, in this case, $\Sigma_0$ is a weakly outermost MOTS whose area satisfies \eqref{eq.equality.area}. Suppose, by contradiction, that $\Sigma_0$ is not a MOTS, \textit{i.e.} $\theta_K^+\le0$ and $\theta_K^+\not\equiv0$ on $\Sigma_0$. Therefore, it follows from \cite[Lemma~5.2]{AndMet} that a small perturbation $\Sigma$ of $\Sigma_0$ in $M$ is such that $\theta_K^+<0$ on $\Sigma$; in particular, $\Sigma$ is homologous to $\Sigma_0$ in $M$.

Now denote by $W$ the compact region bounded by $\Sigma$ and $S$ in $M$ and consider the initial data set $(W,g,-K)$. It follows that the null expansion scalars of $\Sigma$ and $S$ in $(W,g,-K)$ satisfy
\begin{itemize}
\item $\theta_{-K}^+\le0$ on $S$ with respect to the normal that points into $W$;
\item $\theta_{-K}^+>0$ on $\Sigma$ with respect to the normal that points out of $W$.
\end{itemize}
Therefore, by the MOTS Existence Theorem due to L.~Andersson and J.~Metzger \cite{AndMet} in dimension $n=3$ and M.~Eichmair~\cite{Eic1,Eic2} in dimensions $3\le n\le7$ (see \cite[Theorem~2.3]{EicGalMen} for the version that we are using here), there exists an outermost MOTS $\Sigma'$ in $(W,g,-K)$ that is homologous to $\Sigma$ in $W$; in particular, $\Sigma'$ is homologous to $\Sigma_0$ in $M$. Without loss of generality, we may assume that $\Sigma'$ has only homologically nontrivial components.

Now  observe that $\Sigma_0$ is connected, since $M$ is connected and it satisfies the homotopy condition with respect to $\Sigma_0$. Because $H_2(M,\Sigma_0)=0$ and $\Sigma'$ is homologous to $\Sigma_0$, we have that $\Sigma'$ is also connected. Thus, by the Infinitesimal Rigidity (Proposition~\ref{prop.inf.rigidity}), 
\begin{align*}
\Lambda|\Sigma'|\le2\pi\left(1+\sqrt{1-4\Lambda\q(\Sigma')^2}\right).
\end{align*}

On the other hand, $|\Sigma_0|\le|\Sigma'|$, since $\Sigma_0$ minimizes area. Also, because $\div E=0$, the Divergence Theorem gives that $\q(\Sigma')=\q(\Sigma_0)=:\q_0$. Therefore,
\begin{align*}
2\pi\left(1+\sqrt{1-4\Lambda\q_0^2}\right)\le\Lambda|\Sigma_0|\le\Lambda|\Sigma'|\le2\pi\left(1+\sqrt{1-4\Lambda\q_0^2}\right),
\end{align*}
that is,
\begin{align*}
\Lambda|\Sigma'|=\Lambda|\Sigma_0|=2\pi\left(1+\sqrt{1-4\Lambda\q_0^2}\right)
\end{align*}
and the Local Splitting Theorem (Theorem~\ref{thm.local.splitting}) implies that an outer neighborhood of $\Sigma'$ in $(W,g,-K)$ is foliated by MOTS (here we have used that $\Sigma'$ minimizes area, since $|\Sigma'|=|\Sigma_0|$ and $\Sigma_0$ minimizes area), which contradicts the fact that $\Sigma'$ is outermost. This proves that $\Sigma_0$ is a MOTS.

Now we are going to prove that $\Sigma_0$ is weakly outermost. Suppose, by contradiction, that $\Sigma_0$ is not weakly outermost, that is, suppose that there exists a surface $\Sigma$, homologous to $\Sigma_0$ in $M$, such that $\theta_K^+<0$ on $\Sigma$. Without loss of generality, we may assume that each component of $\Sigma$ is homologically nontrivial. Denote by $W$ the compact region bounded by $\Sigma$ and $S$ in $M$ and consider the initial data set $(W,g,-K)$ as before. Then, repeating exactly the same arguments as above, we have a contradiction. Thus $\Sigma_0$ is a weakly outermost MOTS. To get that \eqref{eq.equality.area} holds and finish the proof of the first part of (b), we may apply the Infinitesimal Rigidity to obtain that
\begin{align*}
\Lambda|\Sigma_0|\le2\pi\left(1+\sqrt{1-4\Lambda\q(\Sigma_0)^2}\right),
\end{align*}
since $\Sigma_0$ is weakly outermost and, in particular, stable.

Now, to finish the proof of (b), we are going to prove that, if $4\Lambda\q(\Sigma_0)^2=1$, then $\Lambda|\Sigma_0|=2\pi$. We claim that, in this case, $\Sigma_0$ is a weakly outermost (in particular, stable) MOTS. Therefore, by the Infinitesimal Rigidity, we have $\Lambda|\Sigma_0|=2\pi$. In fact, suppose that $\Sigma_0$ is not a MOTS. Then, the same arguments as before can be applied to ensure the existence of a connected outermost MOTS $\Sigma'$ in $(W,g,-K)$ that is homologous to $\Sigma_0$ in $M$. Then, since $\div E=0$, we have $\q(\Sigma')=\q(\Sigma_0)$, and thus $4\Lambda\q(\Sigma')^2=1$. Therefore, by the Local Splitting Theorem, an outer neighborhood of $\Sigma'$ in $(W,g,-K)$ is foliated by MOTS, which is a contradiction. Analogously we can prove that $\Sigma_0$ is weakly outermost. This finishes the proof of (b).

Now let us prove (a). Assume, by contradiction, that $4\Lambda\q(\Sigma_0)^2>1$. The above arguments give that $\Sigma_0$ is a weakly outermost MOTS as, otherwise, there exists a connected outermost (in particular, stable) MOTS $\Sigma'$ in $(W,g,-K)$ that is homologous to $\Sigma_0$ in $M$. Then, by the Infinitesimal Rigidity,
\begin{align*}
1\ge4\Lambda\q(\Sigma')^2=4\Lambda\q(\Sigma_0)^2>1,
\end{align*} 
which is a contradiction. Thus, if $4\Lambda\q(\Sigma_0)^2>1$, then $\Sigma_0$ is a weakly outermost MOTS, which also contradicts \eqref{eq.charge} in the Infinitesimal Rigidity. This finishes the proof of (a).
\end{proof}

\begin{proof}[Proof of Theorem~\ref{thm.global.splitting}]
First, we observe that $\Sigma_0$ is topologically $S^2$. In fact, if that is not the case, then $\Sigma_0$ is homeomorphic to a torus or to a connected sum of tori. As such, $\Sigma_0$ satisfies what is called the {\it cohomology condition} in \cite{EicGalMen}. The rigidity result Theorem~1.2 in \cite{EicGalMen} then applies so that $0=\mu-|J|\ge\Lambda+|E|^2$, which is a contradiction. Hence $\Sigma_0$ is topologically $S^2$.

On the other hand, Lemma~\ref{main.lemma} gives inequalities \eqref{eq.charge3} and \eqref{eq.area4}; furthermore, if equality holds in \eqref{eq.charge3} or \eqref{eq.area4}, then $\Sigma_0$ is a weakly outermost MOTS whose area satisfies 
\begin{align*}
\Lambda|\Sigma_0|=2\pi\left(1+\sqrt{1-4\Lambda\q_0^2}\right),
\end{align*}
where $\q_0:=\q(\Sigma_0)$. Therefore, we may apply the Local Splitting Theorem to obtain an outer neighborhood $U\cong[0,\delta)\times\Sigma_0$ of $\Sigma_0$ in $M$ such that conclusions (a)-(d) of the theorem hold on $U$. Clearly, $\Sigma_t\cong\{t\}\times\Sigma_0$ converges to a closed embedded MOTS $\Sigma_\delta$ as $t\nearrow\delta$. If $\Sigma_\delta\cap S\neq\varnothing$, the Strong Maximum Principle implies that $\Sigma_\delta=S$. Otherwise, if $\Sigma_\delta\cap S=\varnothing$, we replace $\Sigma_0$ by $\Sigma_\delta$ and $M$ by the complement of $U$ and run the process again (this can be done since $\q(\Sigma_\delta)=\q_0$ and $|\Sigma_\delta|=|\Sigma_0|$). The theorem follows by a continuity argument.
\end{proof}

\bibliographystyle{amsplain}
\bibliography{bibliography.bib}

\providecommand{\bysame}{\leavevmode\hbox to3em{\hrulefill}\thinspace}
\providecommand{\MR}{\relax\ifhmode\unskip\space\fi MR }
\providecommand{\MRhref}[2]{%
  \href{http://www.ams.org/mathscinet-getitem?mr=#1}{#2}
}
\providecommand{\href}[2]{#2}
\begin{thebibliography}{10}

\bibitem{AndMet}
Lars Andersson and Jan Metzger, \emph{The area of horizons and the trapped
  region}, Commun. Math. Phys. \textbf{290} (2009), no.~3, 941--972 (English).

\bibitem{BatLimSou}
Rondinelle~M. Batista, Alexandre~B. Lima, and Paulo~A. Sousa, \emph{Rigidity of
  {MOTS} in charged initial data sets}, preprint.

\bibitem{BraBreNev}
Hubert Bray, Simon Brendle, and Andre Neves, \emph{Rigidity of area-minimizing
  two-spheres in three-manifolds}, Commun. Anal. Geom. \textbf{18} (2010),
  no.~4, 821--830 (English).

\bibitem{CruzLimaSousa}
Tiarlos Cruz, Vanderson Lima, and Alexandre de~Sousa, \emph{Min-max minimal
  surfaces, horizons and electrostatic systems}, 2019, arXiv:1912.08600, to
  appear in J. Diff. Geom.

\bibitem{Eic1}
Michael Eichmair, \emph{The {Plateau} problem for marginally outer trapped
  surfaces}, J. Differ. Geom. \textbf{83} (2009), no.~3, 551--584 (English).

\bibitem{Eic2}
\bysame, \emph{Existence, regularity, and properties of generalized apparent
  horizons}, Commun. Math. Phys. \textbf{294} (2010), no.~3, 745--760
  (English).

\bibitem{EicGalMen}
Michael Eichmair, Gregory~J. Galloway, and Abra{\~a}o Mendes, \emph{Initial
  data rigidity results}, Commun. Math. Phys. \textbf{386} (2021), no.~1,
  253--268 (English).

\bibitem{EicHuaLeeSch}
Michael Eichmair, Lan-Hsuan Huang, Dan~A. Lee, and Richard Schoen, \emph{The
  spacetime positive mass theorem in dimensions less than eight}, J. Eur. Math.
  Soc. (JEMS) \textbf{18} (2016), no.~1, 83--121 (English).

\bibitem{Gal2018}
Gregory~J. Galloway, \emph{Rigidity of outermost {MOTS}: the initial data
  version}, Gen. Relativ. Gravitation \textbf{50} (2018), no.~3, 7 (English),
  Id/No 32.

\bibitem{GalLee}
Gregory~J. Galloway and Dan~A. Lee, \emph{A note on the positive mass theorem
  with boundary}, Lett. Math. Phys. \textbf{111} (2021), no.~4, 10 (English),
  Id/No 111.

\bibitem{GalMen2018}
Gregory~J. Galloway and Abra{\~a}o Mendes, \emph{Rigidity of marginally outer
  trapped {{\(2\)}}-spheres}, Commun. Anal. Geom. \textbf{26} (2018), no.~1,
  63--83 (English).

\bibitem{GalMen2024}
\bysame, \emph{Some rigidity results for compact initial data sets}, Trans.
  Amer. Math. Soc. \textbf{377} (2024), no.~3, 1989–2007 (English).

\bibitem{LeeLesUng2022}
Dan~A. Lee, Martin Lesourd, and Ryan Unger, \emph{Density and positive mass
  theorems for initial data sets with boundary}, Commun. Math. Phys.
  \textbf{395} (2022), no.~2, 643--677 (English).

\bibitem{LeeLesUng}
\bysame, \emph{Density and positive mass theorems for incomplete manifolds},
  Calc. Var. Partial Differ. Equ. \textbf{62} (2023), no.~7, 23 (English),
  Id/No 194.

\bibitem{Loh}
Joachim Lohkamp, \emph{The {H}igher {D}imensional {P}ositive {M}ass {T}heorem
  {II}}, 2016, arXiv:1612.07505.

\end{thebibliography}

\end{document}